\documentclass[a4paper , 11pt]{amsart}

\usepackage[top=30truemm,bottom=35truemm,left=30truemm,right=30truemm]{geometry}
\usepackage{amsmath}
\usepackage{amsthm}
\usepackage{amsfonts}
\usepackage{bm}

\newtheorem{thm}{Theorem}[section]
\newtheorem{lem}[thm]{Lemma}

\newtheorem{cor}[thm]{Corollary}
\theoremstyle{definition}
\newtheorem{prob}[thm]{Problem}
\newtheorem*{rem}{Remark}

\DeclareMathOperator{\insep}{insep}

\DeclareMathOperator{\Res}{Res}

\newcommand{\abs}[1]{\left\lvert#1\right\rvert}

\makeatletter 
\@tempcnta\z@
\loop\ifnum\@tempcnta<26
\advance\@tempcnta\@ne
\expandafter\edef\csname f\@Alph\@tempcnta\endcsname{\noexpand\mathfrak{\@Alph\@tempcnta}}
\expandafter\edef\csname l\@Alph\@tempcnta\endcsname{\noexpand\mathbb{\@Alph\@tempcnta}}
\expandafter\edef\csname c\@Alph\@tempcnta\endcsname{\noexpand\mathcal{\@Alph\@tempcnta}}
\repeat

\title{On quadratic approximation for hyperquadratic continued fractions}
\author[K.~Ayadi]{Khalil Ayadi}
\author[T.~Ooto]{Tomohiro Ooto}
\address{Department of Mathematics Faculty of Sciences, University of Sfax, Tunisia}
\email{ayedikhalil@yahoo.fr}
\address{Tecnos Data Science Engineering, Inc., 27F., Tokyo Opera City, 3-20-2, Nishishinjuku, Shinjuku-ku, Tokyo, 163-1427, Japan}
\email{ooto.tomohiro@gmail.com}
\subjclass[2010]{primary 11J61; secondary 11J68}
\keywords{Diophantine approximation, positive characteristic.}

\begin{document}

\begin{abstract}
	We study quadratic approximations for two families of hyperquadratic continued fractions in the field of Laurent series over a finite field.
	As the first application,  we give the answer to a question of the second author concerning Diophantine exponents for algebraic Laurent series.
	As the second application, we determine the degrees of these families in particular case.
\end{abstract}

\maketitle


\section{Introduction}\label{sec:intro}

Let $n \geq 1$ be an integer and $\xi \in \lR$.
We denote by $w_n(\xi)$ (resp.\ $w_n ^ {*}(\xi)$) the supremum of the real numbers $w$ (resp.\ $w ^ {*}$) which satisfy
\begin{equation*}
	0 < \abs{P(\xi)} \leq H(P)^{- w} \quad (\text{resp.\ } 0 < \abs{\xi - \alpha} \leq H(\alpha)^{- w ^ {*} - 1})
\end{equation*}
for infinitely many integer polynomials $P(X)$ of degree at most $n$ (resp.\ algebraic numbers $\alpha \in \lC$ of degree at most $n$).
Here, $H(P)$ is defined to be the maximum of the absolute values of the coefficients of $P(X)$ and $H(\alpha)$ is equal to $H(Q)$, where $Q(X)$ is the minimal polynomial of $\alpha$ over $\lZ$.
The functions $w_n$ and $w_n ^ {*}$ are called \textit{Diophatine exponents}.
Liouville  proved that $w_1 (\alpha) \leq \deg \alpha - 1$ for any algebraic real number $\alpha$.
Roth \cite{Roth55} improved Liouville Theorem, that is, he showed that $w_1 (\alpha) = 1$ for all algebraic irrational real numbers $\alpha$.
For higher Diophantine exponents, from the Schmidt Subspace Theorem, it is known that
\begin{equation*}
	w_n (\alpha ) = w_n ^ {*} (\alpha ) = \min \{ n, d-1 \},
\end{equation*}
where $n \geq 1$ is an integer and $\alpha$ is an algebraic real number of degree $d$ (see \cite[Section 3]{Bugeaud04}).

Let $p$ be a prime, $q=p^s$, where $s \geq 1$ is an integer and $\lF_q$ be the finite field containing $q$ elements.
We denote by $\lF_q [T], \lF_q (T), $ and $\lF_q ((T ^ {- 1}))$ respectively the ring of polynomials, the field of rational functions, the field of Laurent series in $T ^ {- 1}$ over $\lF_q$.
We consider analogues of Diophantine exponents $w_n$ and $w_n ^ {*}$ for Laurent series (see Section \ref{sec:notation} for the precisely definition).
As an analogue of Liouville Theorem, Mahler \cite{Mahler49} showed $w_1 (\alpha) \leq \deg \alpha - 1$ for any algebraic Laurent series $\alpha \in \lF_q((T^{-1}))$.
However, there exist counter examples of analogue of Roth Theorem for Laurent series over a finite field.
Let $r_1 = p ^ {t_1}$, where $t_1 \geq 1$ is an integer.
He proved that the Laurent series $\alpha_1 := \sum_{n = 0}^{\infty}T ^ {- r_1 ^ n}$ is algebraic of degree $r_1$ with $w_1 (\alpha_1) = r_1 - 1$.
After that many authors investigated rational approximations (e.g.\ \cite{Firicel13, Mathan92, Schmidt00, Thakur99}) and algebraic approximations (e.g.\ \cite{Ooto18, Thakur11, Thakur13}) for algebraic Laurent series.
In particular, the second author (\cite{Ooto18}) proved that, for any rational number $w > 2n-1$, there exists an algebraic Laurent series $\alpha \in \lF_q((T^{-1}))$ such that
\begin{equation*}
	w_1(\alpha) = w_1^{*}(\alpha) = \ldots = w_n(\alpha) = w_n^{*}(\alpha) = w.
\end{equation*}
The first purpose of this paper is to consider Problem 2.2 in \cite{Ooto18}.

\begin{prob}\label{prob:main1}
	Is it true that
	\begin{equation*}
		w_n(\alpha) = w_n^{*}(\alpha)
	\end{equation*}
	for an integer $n \geq 1$ and an algebraic Laurent series $\alpha \in \lF_q((T^{-1}))$?
\end{prob}

Note that it is well-known and easy to see that $w_1(\xi) = w_1^{*}(\xi)$ for all $\xi \in \lF_q((T^{-1}))$.
We give a negative answer of Problem \ref{prob:main1} for any $n \geq 2$ and $q \geq 4$.

\begin{thm}\label{thm:neq_alg}
	Let $n \geq 2$ be an integer, $q = p^s$, and $r = q^t$ with $s, t \geq 1$ be integers and 
	\begin{equation*}
		r \geq \frac{3n + 2 + \sqrt{9n^2 + 4n + 4}}{2}.
	\end{equation*}
	If $q \geq 4$, then there exist algebraic Laurent series $\alpha \in \lF_q ((T ^ {- 1}))$ with $\deg \alpha = r+1$ such that
	\begin{equation*}
		w_n (\alpha) \neq w_n ^ {*} (\alpha)
	\end{equation*}
\end{thm}

Let $\alpha$ be in $\lF_q ((T ^ {- 1}))$ and $r = p ^ t$, where $t \geq 0$ is an integer.
We call $\alpha$ \textit{hyperquadratic} if $\alpha$ is irrational and there exists $(A, B, C, D) \in \lF_q [T] ^ 4 \setminus \{ {\bf 0} \}$ such that
\begin{equation}\label{eq:hyperquad}
	A \alpha^{r+1} + B \alpha ^ r + C \alpha + D = 0.
\end{equation}
For example, quadratic Laurent series are hyperquadratic, $\alpha_1$ is hyperquadratic and satisfies $T \alpha_1^r - T \alpha_1 + 1 = 0$.
Baum and Sweet \cite{Baum76} started a study of hyperquadratic continued fractions in characteristic two.
In 1986, Mills and Robbins \cite{Mills86} showed the existence of hyperquadratic continued fractions with all partial quotient of degree one in odd characteristic with a prime field.
We  refer the reader to \cite{Lasjaunias172} for a survey of recent works of hyperquadratic continued fractions.
We observe that the degree of hyperquadratic Laurent series which satisfies \eqref{eq:hyperquad} is less than or equal to $r+1$.
However, the exact degree of these is open.
In this paper, we discuss the following problem.

\begin{prob}
	Given a hyperquadratic Laurent series $\alpha \in \lF_q((T^{-1}))$, determine the exact degree of $\alpha$.
\end{prob}

In order to determine the exact degree, we approach to quadratic approximations.
Since $w_2(\alpha) \leq \deg \alpha - 1$ for $\alpha$ as in \eqref{eq:hyperquad} (see Lemma \ref{lem:alg}), if $w_2(\alpha) > r-1$, then $\deg \alpha = r+1$.
In this paper, we deal with two families for hyperquadratic continued fractions introduced in \cite{Lasjaunias02, Lasjaunias17}.
We determine the exact degree of these families in particular cases.

This paper is organized as follows.
In Section \ref{sec:notation}, we recall notations which are used in this paper.
In Section \ref{sec:main}, we state main results of quadratic approximation for hyperquadratic continued fractions and applications of the results.
In Section \ref{sec:lemma}, we give lemmas for the proof of main results.
In Section \ref{sec:proof}, we prove the main results.


\section{Notations}\label{sec:notation}

A nonzero Laurent series $\xi \in \lF_q((T ^ {- 1}))$ is represented by $\xi = \sum_{n = N}^{\infty}a_n T ^ {- n}$, where $N \in \lZ, a_n \in \lF_q,$ and $a_N \neq 0$.
The field $\lF_q ((T ^ {- 1}))$ has a non-Archimedean absolute value $\abs{\xi} = q ^ {- N}$ and $\abs{0} = 0$.
The absolute value can be uniquely extended to the algebraic closure of $\lF_q ((T ^ {- 1}))$ and we also write $\abs{\cdot}$ for the extended absolute value.

The {\itshape height} of $P(X) \in (\lF_q [T])[X]$, denoted by $H(P)$, is defined to be the maximum of the absolute values of the coefficients of $P(X)$.
For $\alpha \in \overline{\lF_q (T)}$, there exists a unique non-constant, irreducible, primitive polynomial $P(X) \in (\lF _q [T]) [X]$ whose leading coefficients are monic polynomials in $T$ such that $P(\alpha) = 0$. 
The polynomial $P(X)$ is called the {\itshape minimal polynomial} of $\alpha$.
The {\itshape height} (resp.\ the {\itshape degree}, the {\itshape inseparable degree}) of $\alpha$, denoted by $H(\alpha)$ (resp.\ $\deg \alpha$, $\insep \alpha$), is defined to be the height of $P(X)$ (resp.\ the degree of $P(X)$, the inseparable degree of $P(X)$).
Let $n \geq 1$ be an integer and $\xi$ be in $\lF_q (( T ^ {- 1}))$.
We denote by $w_n (\xi)$ (resp.\ $w_n ^ {*} (\xi)$) the supremum of the real numbers $w$ (resp.\ $w ^ {*}$) which satisfy 
\begin{equation*}
	0 < \abs{P(\xi )} \leq H(P) ^ {- w} \quad (\text{resp.\ } 0 < \abs{\xi - \alpha } \leq H(\alpha) ^ {- w ^ {*} - 1})
\end{equation*}
for infinitely many $P(X) \in (\lF_q [T]) [X]$ of degree at most $n$ (resp.\ $\alpha \in \overline{\lF_q (T)}$ of degree at most $n$).

Let $\alpha \in \overline{\lF_q (T)}$ be a quadratic number.
If $\insep \alpha = 1$, let $\alpha ' \neq \alpha$ be the Galois conjugate of $\alpha$.
If $\insep \alpha = 2$, let $\alpha ' = \alpha $.

A Laurent series $\xi \in \lF_q ((T ^ {- 1}))$ can be expressed as a continued fraction:
\begin{equation*}
	\xi = a_0 + \cfrac{1}{a_1 + \cfrac{1}{a_2 + \cfrac{1}{\cdots }}},
\end{equation*}
where $a_n \in \lF_q [T]$ for all $n \geq 0$ and $\deg _T a_m \geq 1$ for all $m \geq 1$.
We write the expansion $\xi = [a_0, a_1, a_2, \ldots ]$.
The continued fraction expansion of $\xi$ is finite if and only if $\xi$ is in $\lF_q (T)$.
It is known that the continued fraction expansion of $\xi$ is ultimately periodic if and only if $\xi$ is quadratic (see Th\'{e}or\`{e}me 4 in \cite[CHAPITRE IV]{Mathan70}).
We define sequences $(p_n)_{n \geq -1}$ and $(q_n)_{n \geq -1}$ by
\begin{equation*}
	\begin{cases}
		p_{- 1} = 1, \ p_0 = a_0, \ p_n = a_n p_{n - 1} + p_{n - 2},\ n \geq 1,\\
		q_{- 1} = 0,\ q_0 = 1,\ q_n = a_n q_{n - 1} + q_{n - 2},\ n \geq 1.
	\end{cases}
\end{equation*}
We call $(p_n / q_n)_{n \geq 0}$ the {\itshape convergent sequence} of $\xi$ and $p_n / q_n$ the \textit{$n$-th convergent} of $\xi$.

Let  $n \geq 1$ be an integer and $W, V$ be finite words.
We denote by $W \bigoplus V$ the word of concatenation of $W$ and $V$.
We put $W ^ {[n]} := W \bigoplus \cdots \bigoplus W$ ($n$ times) and $\overline{W} := W\bigoplus W\bigoplus \cdots \bigoplus W\bigoplus \cdots $ (infinitely many times).
We write $W ^ {[0]}$ the empty word.

Let $A, B$ be real numbers with $B \neq 0$.
We write $A \ll B$ (resp.\ $A \ll_a B$) if $\abs{A} \leq C \abs{B}$ for some constant (resp.\ some constant depending at most on $a$) $C > 0$.
We write $A \asymp B$ (resp.\ $A \asymp_a B$) if $A \ll B$ and $B \ll A$ (resp.\ $A \ll_a B$ and $B \ll_a A$) hold.


\section{Main results}\label{sec:main}

We recall a family of hyperquadratic continued fractions.
Let $s, t \geq 1$ be integers and we put $q := p ^ s$ and $r := q ^ t$.
Let  $k \geq 0$ be an integer and $\lambda$ be in $\lF_q ^ {*}$.
If $p \neq 2$, then we assume that $\lambda \neq 2$ and put $\mu := 2 - \lambda$.
We define $\Theta_k ^ {t} (\lambda) \in \lF_q ((T ^ {- 1}))$ by
\begin{equation}\label{eq:firstfamily}
	\Theta_k ^ {t} (\lambda)
	=
	\begin{cases}
		[0, T ^ {[k]}, \bigoplus_{i \geq 1} (T, (\lambda T, \mu T) ^ {[(r ^ i - 1) / 2]}) ^ {[k+1]}] & \text{if}\ p \neq 2,\\
		[0, T ^ {[k]} \bigoplus_{i \geq 1} (T, \lambda T ^ {[r ^ i - 1]}) ^ {[k + 1]}] & \text{if}\ p = 2.
	\end{cases}
\end{equation}
Lasjaunias and Ruch \cite{Lasjaunias02} proved that $\Theta_k ^ {t} (\lambda)$ is hyperquadratic and  satisfies the algebraic equation
\begin{equation}\label{eq:firstfamilyequation}
	\begin{cases}
		q_k X ^ {r + 1} - p_k X ^ r + (\lambda \mu ) ^ {(r - 1) / 2} q_{k + r} X - (\lambda \mu) ^ {(r - 1) / 2} p_{k + r} = 0 \quad \text{if}\ p \neq 2,\\
		q_k X ^ {r + 1} - p_k X ^ r + q_{k + r} X - p_{k + r} = 0 \quad \text{if}\ p = 2,
	\end{cases}
\end{equation}
where $(p_n / q_n)_{n \geq 0}$ is the convergent sequence of $\Theta_k ^ {t} (\lambda)$.
Note that if $\lambda = 1$, then $\Theta_k ^ {t} (\lambda)$ is ultimately periodic continued fraction, that is, $\Theta_k ^ {t} (\lambda)$ is quadratic.

Our first main result is the following.

\begin{thm}\label{thm:main1}
	Let $d \geq 2$ be an integer and $\Theta_k ^ {t} (\lambda) \in \lF_q ((T ^ {- 1}))$ be as in \eqref{eq:firstfamily}.
	Assume that $\lambda \neq 1$.
	Then we have
	\begin{gather}\label{eq:w2*}
		w_2 ^ {*} (\Theta_k ^{t} (\lambda)) \geq \max \left\lbrace \frac{r - 1}{k + 1}, r - 1 - \frac{r(r - 1)(r - 4)}{2r(k + 1) + (r - 1)(r - 2)}\right\rbrace ,\\
		\label{eq:w2}
		w_2 (\Theta_k ^ {t} (\lambda)) \geq \max \left\lbrace \frac{r - 1}{k + 1} + 1, r - \frac{r(r - 1)(r - 3)}{2r(k + 1)+(r - 1)(r - 2)}\right\rbrace .
	\end{gather}
	If
	\begin{equation}\label{eq:maincond1}
		k \leq \frac{2(r - 1)}{d + \sqrt{d ^ 2 + 4(2r - 1)d + 4}} - 1,
	\end{equation}
	then we have
	\begin{equation}\label{eq:mainequal1}
		w_n ^ {*} ( \Theta_k ^ {t} (\lambda)) = \frac{r - 1}{k + 1}, \quad w_n (\Theta_k ^ {t} (\lambda)) = \frac{r - 1}{k + 1} + 1
	\end{equation}
	for all integers $2 \leq n \leq d$.
	If
	\begin{equation}\label{eq:maincond2}
		2rd \leq \left( r - 2 - \frac{r(r - 1)(r - 4)}{2r(k + 1) + (r - 1)(r - 2)}\right) \left( r - d - \frac{r(r - 1)(r - 4)}{2r(k + 1) + (r - 1)(r - 2)}\right) ,
	\end{equation}
	then we have
	\begin{gather}\label{eq:mainequal2}
		w_n ^ {*} (\Theta_k ^ {t} (\lambda)) = r - 1 - \frac{r(r - 1)(r - 4)}{2r(k + 1) + (r - 1)(r - 2)},\\
		\label{eq:mainequal22}
		w_n (\Theta_k ^ {t} (\lambda )) = r - \frac{r(r - 1)(r - 3)}{2r(k + 1) + (r - 1)(r - 2)}
	\end{gather}
	for all integers $2 \leq n \leq d$.
\end{thm}

\begin{rem}
	It is well-known and easy to see that for any $\xi \in \lF_q ((T ^ {- 1}))$,
	\begin{equation}\label{eq:w1_determine}
		w_1 (\xi) = w_1 ^ {*} (\xi) = \limsup_{n \longrightarrow \infty } \frac{\deg q_{n + 1}}{\deg q_n},
	\end{equation}
	where $(p_n / q_n)_{n \geq 0}$ is the convergent sequence of $\xi$.
	Since $\Theta_k ^ {t}(\lambda)$ has a bounded partial quotient, we have $w_1 (\Theta_k ^ {t}(\lambda)) = w_1 ^ {*} (\Theta_k ^ {t} (\lambda)) = 1$.
	If
	\begin{equation*}
		r > \frac{3d + 2 + \sqrt{9d ^ 2 + 4d + 4}}{2},
	\end{equation*}
	there exists an effectively computable positive constant $C_1 (r, d)$, depending only on $r$ and $d$, such that we have \eqref{eq:mainequal2} and \eqref{eq:mainequal22} for all $k \geq C_1 (r, d)$.
\end{rem}

The key point of the proof of Theorem \ref{thm:main1} is to find two sequences of very good quadratic approximation, that is, ultimately periodic continued fractions.

As an application of Theorem \ref{thm:main1}, we give a sufficient condition to determine the exact degree of $\Theta_k ^ {t} (\lambda)$.

\begin{cor}\label{cor:main1}
	Let $\Theta_k ^ {t} (\lambda) \in \lF_q ((T ^ {- 1}))$ be as in \eqref{eq:firstfamily}.
	Assume that $\lambda \neq 1$.
	If
	\begin{equation}\label{eq:main1_cor_cond}
		k = 0 \ \text{or}\ k > \frac{(r ^ 2 - 4r + 2)(r - 1)}{2r} - 1,
	\end{equation}
	then we have $\deg \Theta_k ^ {t} (\lambda) = r + 1$.
\end{cor}

We recall another family of hyperquadratic continued fractions.
We define a sequence $(F_n)_{n \geq 0}$ of polynomials in $\lF_q [T]$ by
\begin{equation}\label{eq:fib}
	F_0 = 1, \quad F_1 = T, \quad F_{n + 1} = TF_n + F_{n - 1}\quad \text{for}\ n \geq 1.
\end{equation}
Note that the sequence can be regarded as the analogue of the Fibonacci sequence.
Let $s, t \geq 1$ be integers and we put $q := 2 ^ s, r := 2 ^ t$.
Let $\ell \geq 1$ be an integer, ${\bm \lambda} = (\lambda_1, \ldots , \lambda_{\ell}) \in (\lF_q ^ {*}) ^ {\ell}, {\bm \varepsilon} = (\varepsilon_1, \varepsilon_2) \in (\lF_q ^ {*}) ^ 2$, and $(p_n / q_n)_{0 \leq n \leq \ell - 1}$ be the convergent sequence of the continued fraction $[\lambda_1 T, \ldots , \lambda_{\ell}T]$.
We consider the algebraic equation
\begin{equation}\label{eq:char2eq}
	q_{\ell - 1} X ^ {r + 1} + p_{\ell - 1} X ^ r + (\varepsilon_1 q_{\ell - 2} F_{r-1} + \varepsilon_2 q_{\ell - 1} F_{r - 2}) X + \varepsilon_1 p_{\ell - 2} F_{r - 1} + \varepsilon_2 p_{\ell  - 1} F_{r - 2} = 0.
\end{equation}
Lasjaunias \cite{Lasjaunias17} proved that \eqref{eq:char2eq} has a unique root $\Phi_{\ell} ^ t ({\bm \lambda}; {\bm \varepsilon})$ in $\lF_q ((T ^ {- 1}))$ with $\abs{\Phi_{\ell} ^ t ({\bm \lambda}; {\bm \varepsilon})} \geq q$, $\Phi_{\ell} ^ t ({\bm \lambda}; {\bm \varepsilon})$ is hyperquadratic and its continued fraction expansion is the following:
\begin{equation}\label{eq:secondfamily}
	\Phi_{\ell} ^ t ({\bm \lambda}; {\bm \varepsilon}) = [\lambda_1 T, \lambda_2 T, \ldots ],
\end{equation}
where a sequence $(\lambda_n)_{n \geq 1}$ satisfy
\begin{gather*}
	\lambda_{\ell + r m + 1} = (\varepsilon_2 / \varepsilon_1) \varepsilon_2 ^ {(- 1) ^ {m + 1}} \lambda_{m + 1} ^ r,\\
	\lambda_{\ell + r m + i} = (\varepsilon_1 / \varepsilon_2 ) ^ {(- 1) ^ i} \quad \text{for}\ m \geq 0 \ \text{and}\ 2 \leq i \leq r.
\end{gather*}
Note that if $\varepsilon_1 = \varepsilon_2 = 1$ and $\lambda_j = 1$ for all $1\leq j \leq \ell$, then $\Phi_{\ell} ^ t ({\bm \lambda}; {\bm \varepsilon})$ is ultimately periodic continued fraction, that is, $\Phi_{\ell} ^ t ({\bm \lambda}; {\bm \varepsilon})$ is quadratic.

In this paper, we consider the continued fractions $\Phi_{\ell } ^ t ({\bm \lambda}; {\bm \varepsilon})$ in particular cases.

\begin{thm}\label{thm:main2}
	Let $\Phi_{\ell} ^ t ({\bm \lambda}; {\bm \varepsilon})$ be as in \eqref{eq:secondfamily}, $m \geq 1, d \geq 2$ be integers with $m \leq \ell$, and $\lambda$ be in $\lF_q ^ {*}$ with $\lambda \neq 1$.
	\begin{enumerate}
		\item[(1)] Assume that $\lambda _1 = \lambda , \lambda _i = 1$ for all $2 \leq i \leq m, \varepsilon_1 = \varepsilon_2 = 1$ and if $m < \ell $, then $\lambda_{m + 1} \neq 1$.
		Then we have
		\begin{equation}\label{eq:main22}
			w_2 ^ {*} (\Phi_{\ell} ^ t ({\bm \lambda}; {\bm \varepsilon}))\geq \frac{m}{\ell}(r - 1),\quad 
			w_2 (\Phi_{\ell} ^ t ({\bm \lambda }; {\bm \varepsilon})) \geq 1 + \frac{m}{\ell} (r - 1).
		\end{equation}
		If
		\begin{equation}\label{eq:maincond3}
			\frac{\ell}{m} \leq \frac{2(r - 1)}{d + \sqrt{d ^ 2 + 4(2r - 1)d + 4}},
		\end{equation}
		then we have
		\begin{equation}\label{eq:mainequal3}
			w_n ^ {*} (\Phi_{\ell} ^ t ({\bm \lambda}; {\bm \varepsilon})) = \frac{m}{\ell}(r - 1), \quad 
			w_n (\Phi_{\ell} ^ t ({\bm \lambda}; {\bm \varepsilon})) = 1 + \frac{m}{\ell}(r - 1)
		\end{equation}
		for all integers $2 \leq n \leq d$.
		
		\item[(2)] Assume that $r$ is a power of $q, \varepsilon_1 = \varepsilon_2 = 1$ and $\lambda_i = \lambda$ for all $1 \leq i \leq \ell$.
		Then we have
		\begin{gather}\label{eq:mainequal6}
			w_2 ^ {*} (\Phi_{\ell} ^ t ({\bm \lambda}; {\bm \varepsilon})) \geq r - 1 - \frac{r(r - 1)(r - 4)}{2\ell r + (r - 1)(r - 2)},\\
			\label{eq:mainequal66}
			w_2 (\Phi_{\ell} ^ t ({\bm \lambda}; {\bm \varepsilon})) \geq r - \frac{r(r - 1)(r - 3)}{2\ell r + (r - 1)(r - 2)}.
		\end{gather}
		If
		\begin{equation}\label{eq:maincond7}
			2 r d \leq \left( r - 2 - \frac{r(r - 1)(r - 4)}{2 \ell r + (r - 1)(r - 2)}\right) \left( r - d - \frac{r(r - 1)(r - 4)}{2 \ell r + (r - 1)(r - 2)}\right) ,
		\end{equation}
		then we have
		\begin{gather}\label{eq:mainequal7}
			w_n ^ {*} (\Phi_{\ell} ^ t ({\bm \lambda}; {\bm \varepsilon})) = r - 1 - \frac{r(r - 1)(r - 4)}{2 \ell r + (r - 1)(r - 2)},\\
			\label{eq:mainequal77}
			w_n (\Phi_{\ell} ^ t ({\bm \lambda}; {\bm \varepsilon})) = r - \frac{r (r - 1)(r - 3)}{2 \ell r + (r - 1)(r - 2)}
		\end{gather}
		for all integers $2 \leq n \leq d$.
	\end{enumerate}
\end{thm}

\begin{rem}
	By \eqref{eq:w1_determine}, we have $w_1 (\Phi_{\ell} ^ t ({\bm \lambda}; {\bm \varepsilon})) = w_1 ^ {*} (\Phi_{\ell} ^ t ({\bm \lambda}; {\bm \varepsilon})) = 1$.
	If
	\begin{equation*}
		r > \frac{3d + 2 + \sqrt{9d ^ 2 + 4d + 4}}{2},
	\end{equation*}
	there exists an effectively computable positive constant $C_2 (r, d)$, depending only on $r$ and $d$, such that we have \eqref{eq:mainequal7} and \eqref{eq:mainequal77} for all $\ell \geq C_2 (r, d)$.
\end{rem}

\begin{cor}\label{cor:main2}
	Let $\Phi_{\ell} ^ t ({\bm \lambda}; {\bm \varepsilon})$ be as in \eqref{eq:secondfamily}, $m \geq 1, d \geq 2$ be integers with $m \leq \ell$, and $\lambda$ be in $\lF_q ^ {*}$ with $\lambda \neq 1$.
	\begin{enumerate}
		\item[(1)] Assume that $\lambda _1 = \lambda , \lambda _i = 1$ for all $2 \leq i \leq m, \varepsilon_1 = \varepsilon_2 = 1$ and if $m < \ell $, then $\lambda_{m + 1} \neq 1$.
		If
		\begin{equation*}
			\frac{m}{\ell} > \frac{r - 2}{r - 1},
		\end{equation*}
		then we have $\deg \Phi_{\ell} ^ t ({\bm \lambda}; {\bm \varepsilon}) = r + 1$.
		
		\item[(2)] Assume that $r$ is a power of $q, \varepsilon_1 = \varepsilon_2 = 1$ and $\lambda_i = \lambda$ for all $1 \leq i \leq \ell$.
		If
		\begin{equation*}
			\ell > \frac{(r ^ 2 - 4r + 2)(r - 1)}{2r},
		\end{equation*}
		then we have $\deg \Phi_{\ell} ^ t({\bm \lambda}; {\bm \varepsilon}) = r + 1$.
	\end{enumerate}
\end{cor}

In the last part of this section, we mention a problem associated to Problem \ref{prob:main1}, Corollary \ref{cor:main1} and \ref{cor:main2}.

\begin{prob}
	Let $n\geq 2$ be an integer.
	Determine the set of all values taken by $w_n - w_n ^ {*}$ over the set of algebraic Laurent series.
\end{prob}


\section{Preliminaries}\label{sec:lemma}

In this section, we gather lemmas for the proof of main results.

\begin{lem}\label{lem:alg}{\rm (\cite[Theorem 5.2]{Ooto17}).}
	Let $n \geq 1$ be an integer and $\alpha \in \lF_q ((T ^ {- 1}))$ be an algebraic Laurent series.
	Then we have 
	\begin{equation*}
		w_n ^ {*} (\alpha), w_n (\alpha) \leq \deg \alpha - 1.
	\end{equation*}
\end{lem}

Lemma \ref{lem:contidif} and \ref{lem:conti.lem} are immediately seen.

\begin{lem}\label{lem:contidif}
	Let $\xi = [0, a_1, a_2, \ldots ], \zeta =[0, b_1, b_2, \ldots ]$ be in $\lF_q ((T ^ {- 1}))$.
	Assume that there exists an integer $k \geq 1$ such that $a_n = b_n$ for all $1 \leq n \leq k$ and $a_{k + 1} \neq  b_{k + 1}$.
	Then we have
	\begin{equation*}
	|\xi - \zeta| = \frac{\abs{a_{k + 1} - b_{k + 1}}}{\abs{a_{k + 1} b_{k + 1}}\abs{q_k} ^ 2},
	\end{equation*}
	where $(p_n / q_n)_{n \geq 0}$ is the convergent sequence of $\xi$.
\end{lem}

\begin{lem}\label{lem:conti.lem}
	Let $\xi = [a_0, a_1, \ldots ]$ be in $\lF_q ((T ^ {- 1}))$ and $(p_n / q_n)_{n \geq 0}$ be the convergent sequence of $\xi$.
	Then, for any  $n \geq 1$, we have
	\begin{equation}\label{enum:q}
		|q_n| = |a_1| |a_2| \cdots |a_n|.
	\end{equation}
\end{lem}

\begin{lem}\label{lem:conj2}{\rm (\cite[Lemma 4.6]{Ooto17}).}
	Let $r, s \geq 1$ be integers and $\alpha = [0, a_1,\ldots ,a_r, \overline{a_{r + 1}, \ldots , a_{r + s}}] \in \lF_q ((T ^ {- 1}))$ be an ultimately periodic continued fraction with $a_r \neq a_{r + s}$.
	Let $(p_n / q_n)_{n \geq 0}$ be the convergent sequence of $\alpha$.
	Then we have
	\begin{equation*}
		\frac{\min (|a_r|, |a_{r + s}| )}{|q_r| ^ 2} \leq |\alpha - \alpha '| \leq \frac{|a_r a_{r + s}|}{|q_r| ^ 2}.
	\end{equation*}
\end{lem}

The following lemma is well-known and easy to see.

\begin{lem}\label{lem:quadheightupper}
	Let $r \geq 0, s \geq 1$ be integers and $\alpha = [0, a_1, \ldots , a_r, \overline{a_{r + 1}, \ldots , a_{r + s}}] \in \lF_q ((T ^ {- 1}))$ be an ultimately periodic continued fraction.
	Let $(p_n / q_n)_{n \geq 0}$ be the convergent sequence of $\xi$.
	Then we have $H(\alpha) \leq |q_r q_{r + s}|$.
\end{lem}

\begin{lem}\label{lem:quadheight}
	Let $b, c, d \in \lF_q [T]$ be distinct polynomials with $\deg_T b, \deg_T c, \deg_T d \geq 1$.
	Let $n, m, \ell \geq 1$ be integers and $a_1, \ldots , a_{n - 1} \in \lF_q [T]$ be polynomials with $\deg_T a_i \geq 1$ for all $1 \leq i \leq n - 1$.
	We put
	\begin{gather*}
		\alpha_1 := [0, a_1, \ldots , a_{n - 1}, c, \overline{b}], \quad 
		\alpha_2 := [0, a_1, \ldots , a_{n - 1}, c, \overline{b ^ {[m]}, d}],\\
		\alpha_3 := [0, a_1, \ldots , a_{n - 1}, c, \overline{b ^ {[m]}, c, b ^ {[\ell]}}], \quad
		\alpha_4 := [0, a_1, \ldots , a_{n - 1}, c, \overline{(b, d) ^ {[m]}, c, (b, d) ^ {[\ell]}}].
	\end{gather*}
	Let $(p_{i, k} / q_{i, k})_{k \geq 0}$ be the convergent sequence of $\xi_i$ for $i = 1, 2, 3, 4$.
	Then we have
	\begin{gather}\label{eq:quadheight12}
		H(\alpha_1) \asymp_{b, c} \abs{q_{1, n}} ^ 2, \quad
		H(\alpha_2) \asymp _{b, c, d} \abs{q_{2, n} q_{2, n + m + 1}},\\
		\label{eq:quadheight3}
		H(\alpha_3) \asymp_{b, c} \abs{q_{3, n} q_{3, n + m + \ell + 1}}, \quad
		H(\alpha_4) \asymp_{b, c, d} \abs{q_{4, n} q_{4, n + 2m + 2\ell + 1}}.
	\end{gather}
\end{lem}

\begin{proof}
	See Lemma 4.7 in \cite{Ooto17} for the proof of \eqref{eq:quadheight12}.
	
	It follows from Lemma \ref{lem:quadheightupper} that $H(\alpha_3) \leq |q_{3, n} q_{3, n + m + \ell + 1}|$.
	We put 
	\begin{equation*}
		\beta_3 :=[0, a_1, \ldots , a_{n - 1}, c, b ^ {[m]}, c, \overline{b}].
	\end{equation*}
	
	Let $P_3(X)$ and $Q_3(X)$ be the minimal polynomial of $\alpha_3$ and $\beta_3$, respectively.
	Since $P_3$ and $Q_3$ do not have a common root, we have
	\begin{equation*}
		1 \leq |\Res(P_3, Q_3)| \leq H(\beta_3) ^ 2 H(\alpha_3) ^ 2 |\alpha_3 - \beta_3| |\alpha_3' - \beta_3| |\alpha_3 - \beta_3'| |\alpha_3' - \beta_3'|.
	\end{equation*}
	By Lemma \ref{lem:contidif}, \ref{lem:conti.lem} and \ref{lem:conj2}, we obtain
	\begin{gather*}
		|\alpha_3 - \beta_3| \ll_{b, c} |q_{3, n + 2m + \ell + 1}| ^ {- 2}, \quad 
		|\alpha_3 - \beta_3'| \ll_{b, c} |q_{3, n + m + 1}| ^ {- 2},\\
		|\alpha_3' - \beta_3|, |\alpha_3' - \beta_3'| \ll_{b, c} |q_{3, n}| ^ {- 2}.
	\end{gather*}
	Therefore, by Lemma \ref{lem:conti.lem} and \ref{lem:quadheightupper}, we have $H(\alpha_3) \gg_{b, c} |q_{3, n} q_{3, n + m + \ell + 1}|$.
	
	In the same way to the above proof, we obtain $H(\alpha_4) \asymp_{b, c, d} \abs{q_{4, n} q_{4, n + 2m + 2\ell + 1}}$.
\end{proof}

\begin{lem}\label{lem:wnwn*}{\rm (\cite[Proposition 5.6]{Ooto17}).}
	Let $n \geq 1$ be an integer and $\xi \in \lF_q ((T ^ {- 1}))$.
	Then we have 
	\begin{equation*}
		w_n ^ {*} (\xi) \leq w_n (\xi).
	\end{equation*}
\end{lem}

\begin{lem}\label{lem:wnlower}
	Let $n, m \geq 1$ be integers with $n \geq m$ and $\xi \in \lF_q ((T ^ {- 1}))$ be not algebraic of degree at most $m$.
	Then we have $w_n (\xi) \geq m$.
\end{lem}

\begin{proof}
	Since the sequence $(w_k (\xi))_{k \geq 1}$ is increasing, we have $w_n (\xi) \geq m$ (see e.g.\ \cite[page~200]{Bugeaud04}).
\end{proof}

The following lemma is well-known and immediately seen.

\begin{lem}\label{lem:height}
	Let $P(X)$ be in $(\lF_q [T]) [X]$.
	Assume that $P(X)$ can be factorized as
	\begin{equation*}
		P(X) = A \prod_{i = 1}^{n} (X - \alpha_i),
	\end{equation*}
	where $A \in \lF_q [T]$ and $\alpha_i \in \overline{\lF_q (T)}$ for $1 \leq i \leq n$.
	Then we have
	\begin{equation*}
		H(P) = |A| \prod_{i = 1}^{n} \max (1, |\alpha_i|).
	\end{equation*}
\end{lem}

The following lemma is immediately seen by the definition of discriminant.

\begin{lem}\label{lem:Galois}
	Let $\alpha \in \overline{\lF_q (T)}$ be a quadratic number.
	If $\alpha \neq \alpha'$, then we have
	\begin{equation*}
		|\alpha - \alpha'| \geq H(\alpha) ^ {- 1}.
	\end{equation*}
\end{lem}

The following two lemmas are key lemmas for proving Theorem \ref{thm:main1} and \ref{thm:main2}.

\begin{lem}\label{lem:bestquad}{\rm (\cite[Lemma 3.19]{Ooto18}).}
	Let $d \geq 2$ be an integer.
	Let $\xi$ be in $\lF_q ((T ^ {- 1}))$, $\theta , \delta$ be positive numbers, and $\varepsilon$ be a non-negative number.
	Assume that there exist a sequence $(\alpha_j)_{j \geq 1}$ and positive number $c$ such that for any $j \geq 1$, $\alpha_j \in \overline{\lF_q (T)}$ is quadratic with $0 < |\alpha_j - \alpha_j'| \leq c$ and $\xi \neq \alpha_j$,  $(H(\alpha_j))_{j \geq 1}$ is a divergent increasing sequence, and
	\begin{gather*}
		\limsup_{k \rightarrow \infty } \frac{\log H(\alpha_{k + 1})}{\log H(\alpha_k)} \leq \theta ,\\
		\lim_{k \rightarrow \infty } \frac{- \log |\xi - \alpha_k|}{\log H(\alpha_k)} = d + \delta , \quad \lim_{k \rightarrow \infty } \frac{- \log |\alpha_k - \alpha_k'|}{\log H(\alpha_k)} = \varepsilon .
	\end{gather*}
	If $2 d \theta \leq (d - 2 + \delta ) \delta $, then we have for all $2 \leq n \leq d$,
	\begin{equation}\label{eq:bestquad_eq}
		w_n ^ {*} (\xi) = d - 1 + \delta , \quad w_n(\xi) = d - 1 + \delta + \varepsilon .
	\end{equation}
\end{lem}

\begin{lem}\label{lem:bestquad2}
	Let $\xi$ be in $\lF_q ((T ^ {- 1})) \setminus \lF_q (T)$.
	Assume that there exists a sequence $(\alpha_j)_{j \geq 1}$ such that for any $j \geq 1$, $\alpha_j \in \overline{\lF_q (T)}$ is quadratic with $\alpha_j \neq \alpha_j'$ and $\xi \neq \alpha_j$, and $(H(\alpha_j))_{j \geq 1}$ is a divergent increasing sequence.
	If there exist limits of the sequences 
	\begin{equation*}
		\left( \frac{- \log |\xi - \alpha_j|}{\log H(\alpha_j)}\right) _{j \geq 1}\ \text{and}\ \left( \frac{- \log |\alpha_j - \alpha_j'|}{\log H(\alpha_j)}\right) _{j \geq 1},
	\end{equation*}
	then we have
	\begin{gather}\label{eq:quad1}
		w_2 ^ {*} (\xi) \geq \lim_{j \rightarrow \infty } \frac{- \log |\xi - \alpha_j|}{\log H(\alpha_j)} - 1,\\
		\label{eq:quad2}
		w_2 (\xi) \geq \lim_{j \rightarrow \infty } \frac{- \log |\xi - \alpha_j|}{\log H(\alpha_j)} + \lim_{j \rightarrow \infty }\frac{- \log |\alpha_j - \alpha_j'|}{\log H(\alpha_j)} - 1.
	\end{gather}
\end{lem}

\begin{proof}
	The inequality \eqref{eq:quad1} is immediately seen.
	In what follows, we show \eqref{eq:quad2}.
	We put
	\begin{equation*}
		\gamma_1 := \lim_{j \rightarrow \infty } \frac{- \log |\xi - \alpha_j|}{\log H(\alpha_j)}, \quad \gamma_2 := \lim_{j \rightarrow \infty } \frac{- \log |\alpha_j - \alpha_j'|}{\log H(\alpha_j)}.
	\end{equation*}
	
	We may assume that $\gamma_1 > 1 $ and $\gamma_2 > 0$.
	In fact, if $\gamma_2 \leq 0$, then we have \eqref{eq:quad2} by Lemma \ref{lem:wnwn*}.
	By Lemma \ref{lem:Galois}, we have $\gamma_2 \leq 1$, which implies $\gamma_2 < \gamma_1$.
	If $\gamma_1 \leq 1$, then we obtain \eqref{eq:quad2} by Lemma \ref{lem:wnlower}.
	
	Then we obtain
	\begin{equation*}
		|\xi - \alpha_j| < |\alpha_j - \alpha_j'| \leq 1
	\end{equation*}
	for all sufficiently large $j$.
	Therefore, we have
	\begin{gather*}
		\max (1, |\xi|) = \max (1, |\alpha_j|) = \max (1, |\alpha_j'|),\\
		|\xi - \alpha_j'| = |\alpha_j - \alpha_j'|
	\end{gather*}
	for all sufficiently large $j$.
	It follows from Lemma \ref{lem:height} that
	\begin{equation*}
		H(P_j) = |A_j| \max (1, |\xi|) ^ 2
	\end{equation*}
	for all sufficiently large $j$, where $P_j (X) = A_j (X - \alpha_j) (X - \alpha_j ')$ is the minimal polynomial of $\alpha_j$.
	Hence, we have
	\begin{equation*}
		|P_j (\xi)| = H(P_j) |\xi - \alpha_j| |\alpha_j - \alpha_j'| \max (1, |\xi|) ^ {- 2}
	\end{equation*}
	for all sufficiently large $j$.
	Thus, we obtain
	\begin{equation*}
		\lim_{j\rightarrow \infty }\frac{- \log |P_j (\xi)|}{\log H(P_j)} = \gamma_1 + \gamma_2 - 1,
	\end{equation*}
	which implies \eqref{eq:quad2}.
\end{proof}


\section{Proof of Main results}\label{sec:proof}

In this section, we prove the main results, that is, Theorem \ref{thm:neq_alg}, \ref{thm:main1}, \ref{thm:main2}, Corollary \ref{cor:main1} and \ref{cor:main2}.

\begin{proof}[Proof of Theorem \ref{thm:main1}]
	For an integer $n \geq 3$, we put
	\begin{gather*}
		\alpha_n :=
		\begin{cases}
			[0, T ^ {[k]}, \bigoplus_{1 \leq i \leq n - 1}(T, (\lambda T, \mu T) ^ {[(r ^ i - 1) / 2]}) ^ {[k + 1]}, T, \overline{\lambda T, \mu T}] & \text{if}\ p \neq 2,\\
			[0, T ^ {[k]}, \bigoplus_{1 \leq i \leq n - 1}(T, \lambda T ^ {[r ^ i - 1]}) ^ {[k + 1]}, T, \overline{\lambda T}] & \text{if}\ p = 2,
		\end{cases}\\
		\beta_n :=
		\begin{cases}
			[0, T ^ {[k]}, \bigoplus_{1 \leq i \leq n - 1} (T, (\lambda T, \mu T) ^ {[(r ^ i - 1) / 2]}) ^ {[k + 1]}, \overline{T, (\lambda T, \mu T) ^ {[(r ^ n - 1) / 2]}}] & \text{if}\ p \neq 2,\\
			[0, T ^ {[k]}, \bigoplus_{1 \leq i \leq n - 1} (T, \lambda T ^ {[r ^ i - 1]}) ^ {[k + 1]}, \overline{T, \lambda T ^ {[r ^ n - 1]}}] & \text{if}\ p = 2.
		\end{cases}
	\end{gather*}
	Then we have
	\begin{align*}
		\beta_n =
		[0, T ^ {[k]}, \bigoplus_{1 \leq i \leq n - 2}(T, (\lambda T, & \mu T) ^ {[(r ^ i - 1) / 2]}) ^ {[k + 1]}, (T, (\lambda T, \mu T) ^ {[(r ^ {n - 1} - 1) / 2]})^{[k]},\\
		& T, \overline{(\lambda T, \mu T) ^ {[(r ^ {n - 1} - 1) / 2]}, T, (\lambda T, \mu T) ^ {[(r ^ n - r ^ {n - 1}) / 2]}}]	
	\end{align*}
	if $p \neq 2$, and
	\begin{align*}
		\beta_n	=
		[0, T ^ {[k]}, \bigoplus_{1 \leq i \leq n - 2} (T, \lambda T ^ {[r ^ i - 1]}) ^ {[k + 1]}, (T, \lambda T ^ {[r ^ {n - 1} - 1]}) ^ {[k]}, T, \overline{\lambda T ^ {[r ^ {n - 1} - 1]}, T, \lambda T ^ {[r ^ n - r ^ {n - 1}]}}]
	\end{align*}
	if $p = 2$.
	It follows from Lemma \ref{lem:conti.lem} and \ref{lem:quadheight} that
	\begin{equation*}
		H(\alpha_n) \asymp q ^ {2 (k + 1) (r ^ n - 1) / (r - 1)}, \quad
		H(\beta_n) \asymp q ^ {2 (k + 1) (r ^ n - 1) / (r - 1) + r ^ n - 2 r ^ {n - 1}}.
	\end{equation*}
	Since $\Theta_k ^ {t} (\lambda)$ and $\alpha_n$ have the same first $((k + 1) \sum_{i = 0} ^ {n - 1} r ^ i + r ^ n - 1)$-th partial quotients, while the next partial quotient are different, we have
	\begin{equation*}
		\abs{\Theta_k ^ {t} (\lambda) - \alpha_n} \asymp q ^ {- 2 (k + 1) (r ^ n - 1) / (r - 1) - 2 r ^ n}
	\end{equation*}
	by Lemma \ref{lem:contidif} and \ref{lem:conti.lem}.
	Similary, since $\Theta_k ^ {t} (\lambda)$ and $\beta_n$ have the same first $((k + 1) \sum_{i = 0} ^ {n} r ^ i + r ^ n - 1)$-th partial quotients, while the next partial quotient are different, we get
	\begin{equation*}
		\abs{\Theta_k ^ {t} (\lambda) - \beta_n} \asymp q ^ {- 2 (k + 1) (r ^ {n + 1} - 1) / (r - 1) - 2 r ^ n}.
	\end{equation*}
	By Lemma \ref{lem:conti.lem} and \ref{lem:conj2}, we obtain
	\begin{equation*}
		\abs{\alpha_n - \alpha_n'} \asymp q ^ {- 2 (k + 1) (r ^ n - 1) / (r - 1)},\quad
		\abs{\beta_n - \beta_n'} \asymp q ^ {- 2 (k + 1) (r ^ n - 1) / (r - 1) + 2 r ^ {n - 1}}.
	\end{equation*}
	Therefore, we deduce that
	\begin{gather*}
		\lim_{n \rightarrow \infty } \frac{\log H(\alpha_{n + 1})}{\log H(\alpha_n)} = \lim_{n \rightarrow \infty } \frac{\log H(\beta_{n + 1})}{\log H(\beta_n)} = r,\\
		\lim_{n \rightarrow \infty } \frac{- \log \abs{\Theta_k ^ {t} (\lambda) - \alpha_n}}{\log H(\alpha_n)} = 1 + \frac{r - 1}{k + 1},\quad 
		\lim_{n \rightarrow \infty } \frac{- \log \abs{\alpha_n - \alpha_n'}}{\log H(\alpha_n)} = 1,\\
		\lim_{n \rightarrow \infty } \frac{- \log \abs{\Theta_k ^ {t} (\lambda) - \beta_n}}{\log H(\beta_n)} = r - \frac{r (r - 1)(r - 4)}{2 r (k + 1) + (r - 1) (r - 2)},\\
		\lim_{n \rightarrow \infty } \frac{- \log \abs{\beta_n - \beta_n'}}{\log H(\beta_n)} = 1 - \frac{r (r - 1)}{2 r (k + 1) + (r - 1)(r - 2)}.
	\end{gather*}
	Hence, we obtain \eqref{eq:w2*} and \eqref{eq:w2} by Lemma \ref{lem:bestquad2}.
	
	If \eqref{eq:maincond1} holds, then we heve
	\begin{equation*}
		2 r d \leq \left( \frac{r - 1}{k + 1} - 1\right) \left( \frac{r - 1}{k + 1} - d + 1 \right) .
	\end{equation*}
	Therefore, by Lemma \ref{lem:bestquad}, we obtain \eqref{eq:mainequal1}.
	
	Similaly, if \eqref{eq:maincond2} holds, then we heve \eqref{eq:mainequal2} and \eqref{eq:mainequal22} by Lemma \ref{lem:bestquad}.
\end{proof}

\begin{proof}[Proof of Corollary \ref{cor:main1}]
	It follows from \eqref{eq:hyperquad} that $\deg \Theta_k ^ {t} (\lambda) \leq r+1$.
	If \eqref{eq:main1_cor_cond} holds, then we have $w_2(\Theta_k ^ {t} (\lambda)) > r-1$ by Theorem \ref{thm:main1}.
	Therefore, by Lemma \ref{lem:alg}, we obtain $\deg \Theta_k ^ {t} (\lambda) = r+1$.
\end{proof}

\begin{proof}[Proof of Theorem \ref{thm:neq_alg}]
	It follows from $q \geq 4$ that we can take $\lambda \in \lF_q ^ {*}$ with $\lambda \neq 1$ and $\lambda \neq 2$.
	Let $n \geq 2$ be integers.
	Since $r \geq (3n + 2 + \sqrt{9n^2 + 4n + 4})/2$, we take $r = p^t$, where $t \geq 0$ is an integer with
	\begin{equation*}
		0 \leq \frac{2(r - 1)}{n + \sqrt{n ^ 2 + 4(2r - 1)n + 4}} - 1.
	\end{equation*}
	Then, by Theorem \ref{thm:main1} and Corollary \ref{cor:main1}, we have $w_n(\Theta_0 ^ {t} (\lambda)) \neq w_n^{*}(\Theta_0 ^ {t} (\lambda))$ and $\deg \Theta_0 ^ {t} (\lambda) = r + 1$.
\end{proof}

\begin{proof}[Proof of Theorem \ref{thm:main2}]
	Since $w_n (\xi) = w_n (\xi^{- 1})$ and $w_n ^ {*} (\xi) = w_n ^ {*} (\xi ^{- 1})$ for all $0 \neq \xi \in \lF_q ((T ^ {- 1}))$ and $n \geq 1$, we consider $\Phi := \Phi_{\ell} ^ t({\bm \lambda}; {\bm \varepsilon}) ^ {- 1} = [0, \lambda_1T, \lambda_2T, \ldots ]$ instead of $\Phi_{\ell} ^ t ({\bm \lambda}; {\bm \varepsilon})$.
	By the assumption, we obtain
	\begin{equation*}
		\lambda_n
		=
		\begin{cases}
			\lambda_{i + 1} ^ {r ^ j} & \text{if}\ n = 1 + \ell \sum_{h = 0} ^ {j - 1} r ^ h + i r ^ j \ \text{for some}\ j \geq 0, 0 \leq i < \ell ,\\
			1 & \text{otherwise}.
		\end{cases}
	\end{equation*}
	
	First, we prove (1).
	For an integer $n \geq 1$, we put
	\begin{equation*}
		\alpha_n := [0, \lambda_1T, \lambda_2T, \ldots ,\lambda_{i(n)}T, \overline{T}],
	\end{equation*}
	where $i(n) = 1 + \ell \sum_{j = 0} ^ {n - 1} r ^ j$.
	Then we have $\lambda_{i(n)} = \lambda ^ {r ^ n}$.
	In a similar way to the proof of Theorem \ref{thm:main1}, we obtain
	\begin{gather*}
	\abs{\Phi - \alpha_n} \asymp q ^ {- 2 \ell (r ^ n - 1) / (r - 1) - 2 m r ^ n},\quad 
	\abs{\alpha_n - \alpha_n'} \asymp q ^ {- 2 \ell (r ^ n - 1) / (r - 1)},\\
	H(\alpha_n) \asymp q ^ {2 \ell (r ^ n - 1) / (r - 1)}.
	\end{gather*}
	Therefore, we have \eqref{eq:main22} by Lemma \ref{lem:bestquad2}.
	If \eqref{eq:maincond3} holds, then we obtain \eqref{eq:mainequal3} by Lemma \ref{lem:bestquad}.
	
	Next, we prove (2).
	For an integer $n \geq 1$, we put
	\begin{equation*}
	\beta_n := [0, \lambda_1T, \lambda_2T, \ldots ,\lambda_{i(n) - 1}T, \overline{\lambda ^ {r ^ n}T, T ^ {[r ^ n - 1]}}],
	\end{equation*}
	Then we have
	\begin{equation*}
	\beta_n = [0, \lambda_1T, \lambda_2T, \ldots ,\lambda_{i(n) - r ^ {n - 1}}T, \overline{T ^ {[r ^ {n - 1} - 1]},\lambda T, T ^ {[r ^ n - r ^ {n - 1}]}}],
	\end{equation*}
	and $\lambda_{i(n) - r ^ {n - 1}}=\lambda$.
	In a similar way to the proof of Theorem \ref{thm:main1}, we obtain
	\begin{gather*}
	\abs{\Phi - \beta_n} \asymp q ^ {- 2 \ell (r ^ n - 1) / (r - 1) - 2 m r ^ n},\quad 
	\abs{\beta_n - \beta_n'} \asymp q ^ {- 2 \ell (r ^ n - 1) / (r - 1) + 2 r ^ {n - 1}},\\
	H(\beta_n) \asymp q ^{2 \ell (r ^ n - 1) / (r - 1) + r ^ n - 2 r ^ {n - 1}}.
	\end{gather*}
	Therefore, we have \eqref{eq:mainequal6} and \eqref{eq:mainequal66} by Lemma \ref{lem:bestquad2}.
	If \eqref{eq:maincond7} holds, then we obtain \eqref{eq:mainequal7} and \eqref{eq:mainequal77} by Lemma \ref{lem:bestquad}.
\end{proof}

\begin{proof}[Proof of Corollary \ref{cor:main2}]
	In the same way to the proof Corollary \ref{cor:main1}, we prove Corollary \ref{cor:main2}.
\end{proof}

\subsection*{Acknowledgements}
The author would like to thank the referee for helpful comments.


\end{document}